\documentclass[letterpage, 11pt, notitlepage]{article}

\usepackage[margin=1.0in]{geometry}

\usepackage{amsmath, amsfonts, amssymb, graphicx,epstopdf,color,soul,mathtools,enumitem,amsthm} 

\usepackage{fullpage}
\usepackage{epsfig}
\usepackage{wrapfig}
\usepackage{graphics,color}
\usepackage{tabularx}
\usepackage{algorithmic}
\usepackage{lipsum}
\usepackage{cite}
\usepackage{url}
\usepackage{multicol}
\usepackage{subfigure}
\usepackage{setspace}
\usepackage{tikz,lipsum,lmodern}
\usepackage[most]{tcolorbox}
\usepackage{bbm}
\usepackage{bm}
\usepackage{palatino}
\usepackage{floatflt}
\usepackage{titlesec}
\allowdisplaybreaks

\usepackage[bookmarks=true,pageanchor,colorlinks,linkcolor=red,anchorcolor=blue, citecolor=blue,urlcolor=blue,hyperfootnotes=false]{hyperref}

\titlespacing\paragraph{0pt}{4pt plus 2pt minus 2pt}{4pt plus 2pt minus 2pt}

\newcommand{\Eset}{\mathbb{E}}

\newcommand{\Rset}{\mathbb{R}}

\newcommand{\Acal}{{\cal A}}
\newcommand{\Bcal}{{\cal B}}

\newcommand{\Fcal}{{\cal F}}

\newcommand{\Mcal}{{\cal M}}

\newcommand{\Ocal}{{\cal O}}
\newcommand{\Pcal}{{\cal P}}
\newcommand{\Qcal}{{\cal Q}}

\newcommand{\Scal}{{\cal S}}

\newcommand{\Abf}{{\bf A}}

\newcommand{\xhat}{{\hat{x}}}
\newcommand{\yhat}{{\hat{y}}}

\newtheorem{lem}{Lemma}
\newtheorem{thm}{Theorem}

\newtheorem{assump}{Assumption}
\newtheorem{remark}{Remark}

\date{}

\title{\LARGE Fast Nonlinear Two-Time-Scale Stochastic Approximation: Achieving $\Ocal(1/k)$ Finite-Sample Complexity}
\author{
Thinh T. Doan\thanks{Thinh T. Doan is with the Bradley Department of Electrical and Computer Engineering, Virginia Tech, USA. Email: {\tt\small thinhdoan@vt.edu}}
}

\begin{document}

\maketitle
\begin{abstract}
This paper proposes to develop a new variant of the two-time-scale stochastic approximation to find the roots of two coupled nonlinear operators, assuming only noisy samples of these operators can be observed. Our key idea is to leverage the classic Ruppert–Polyak averaging technique to dynamically estimate the operators through their samples. The estimated values of these averaging steps will then be used in the two-time-scale stochastic approximation updates to find the desired solution. Our main theoretical result is to show that under the strongly monotone condition of the underlying nonlinear operators the mean-squared errors of the iterates generated by the proposed method converge to zero at an optimal rate $\Ocal(1/k)$, where $k$ is the number of iterations. Our result significantly improves the existing result of two-time-scale stochastic approximation, where the best known finite-time convergence rate is $\Ocal(1/k^{2/3})$. We illustrate this result by applying the proposed method to develop new reinforcement learning algorithms with improved performance.  
\end{abstract}


\section{Introduction}\label{sec:intro}
In this paper, we study the problem of finding the root of two coupled nonlinear operators. In particular, given two operators $F:\Rset^{d}\times\Rset^{d}\rightarrow\Rset^{d}$ and $G:\Rset^{d}\times\Rset^{d}\rightarrow\Rset^{d}$, we seek to find $x^{\star}$ and $y^{\star}$ such that
\begin{align}
\left\{
\begin{aligned}
&F(x^{\star},y^{\star}) = 0\\
&G(x^{\star},y^{\star}) = 0.
\end{aligned}\right.\label{prob:FG}
\end{align}
We will be interested in the setting where $F,G$ are unknown. Instead, we assume that there is a stochastic oracle that outputs noisy values of $F(x,y)$ and $G(x,y)$ for a given pair $(x,y)$, i.e.,  for any $x$ and $y$ we have access to $F(x,y)\, +\, \xi$ and $G(x,y)\, +\, \psi$, where $\xi$ and $\psi$ are two random variables. For solving problem \eqref{prob:FG} in this stochastic setting, two-time-scale stochastic approximation (SA) \cite{borkar2008}, a generalized variant of the classic SA \cite{RobbinsM1951}, is a natural choice. This method iteratively updates the variables $x_{k}$ and $y_{k}$, the estimates of $x^{\star}$ and $y^{\star}$, respectively, for all $k\geq 0$ as
\begin{align}
\begin{aligned}
x_{k+1} &= x_{k} - \alpha_{k}\left(F(x_{k},y_{k}) + \xi_{k}\right)\\   
y_{k+1} &= y_{k} - \beta_{k}\left(G(x_{k},y_{k}) + \psi_{k}\right),
\end{aligned}\label{alg:xy}
\end{align}
where $x_{0}$ and $y_{0}$ are arbitrarily initialized in $\Rset^{d}$. In \eqref{alg:xy}, $\alpha_{k}$ and $\beta_{k}$ are two nonnegative step sizes chosen such that $\beta_{k}\ll \alpha_{k}$, i.e., the second iterate is updated using step sizes that are very small as compared to the ones used to update the first iterate.  Thus, the update of $x_{k}$ is referred to as the ``fast-time scale" while the update of $y_{k}$ is called the ``slow-time scale", explaining the name of two-time-scale SA. The time-scale difference here is loosely defined as the ratio between the two step sizes, i.e., $\beta_{k}/\alpha_{k}$. In addition, the update of the \textit{fast iterate} depends on the \textit{slow iterate} and vice versa, that is, they are coupled to each other. To handle this coupling, the two step sizes have to be properly chosen to guarantee the convergence of the method. Indeed, an important problem in this area is to properly select the two step sizes so that the two iterates converge as fast as possible. This task, however, is challenging due to the coupling between the iterates and the impact of sampling noise on their updates. 

At present, the asymptotic convergence of two-time-scale SA, i.e., whether $(x_{k},y_{k})$ converges to $(x^{\star},y^{\star})$, is well-studied by using the method of ordinary differential equation (ODE). This result can be found in several textbooks, see for example \cite{borkar2008,benveniste2012adaptive}. The asymptotic convergence rates of this method have been also studied via the central limit theorem in both linear and nonlinear settings \cite{KondaT2004,MokkademP2006,hu2024central}. In particular, these works show that two-time-scale SA asymptotically converges in distribution at a rate $\Ocal(1/k)$. On the other hand, the finite-sample complexity (or finite-time convergence rate) of this method in the mean-squared regime is rather incomplete. In the linear setting, i.e., when $F,G$ are linear on $x,y$, the mean-square error of the iterates generated by \eqref{alg:xy} converges to zero at a rate $\Ocal(1/k)$ \cite{Kaledin_two_time_SA2020,haque2023tight}. When $F,G$ are nonlinear, only a suboptimal convergence rate, $\Ocal(1/k^{2/3})$, is known \cite{Doan_two_time_SA2020}.  Motivated by the work in \cite{Doan_two_time_SA2020}, the authors in \cite{han2024finite} show that the two-time-scale SA in \eqref{alg:xy} converges at a rate $\Ocal(1/k)$, however, this result requires a restrictive condition on the high order of the smoothness of the operators (see the discussion in Remark \ref{remark:thm} for more detail). This leaves the question of whether the two-time-scale SA in \eqref{alg:xy} can find the solution $(x^{\star},y^{\star})$ of \eqref{prob:FG} with an $\Ocal(1/k)$ finite-time sample complexity. The overarching goal of this paper is to (partially) answer this question, where our main contribution is:\vspace{0.1cm}

\noindent\textbf{Main Contributions.} \textit{This paper proposes to develop a new variant of the two-time-scale SA in \eqref{alg:xy}, leveraging the classic Ruppert-Polyak averaging technique. The mean square errors of the iterates generated by the proposed methods are shown to converge to zero at a rate $\Ocal(1/k)$, which significantly improves the best known result $\Ocal(1/k^{2/3})$ in the setting of strongly monotone and Lipschitz continuous nonlinear operators. We will also leverage the proposed method to develop new reinforcement learning algorithms with improved performance.} 


\subsection{Prior Works}
\noindent\textbf{Stochastic Approximation (SA).} The two-time-scale SA in \eqref{alg:xy} is a generalized variant of the classic SA introduced by \cite{RobbinsM1951} with broad applications in various areas including statistics, stochastic optimization, machine learning, and reinforcement learning \cite{BTbook1999,borkar2008,TTFBook2009,SBbook2018,LanBook2020}. In general, SA is a simulation-based approach for finding the root (or fixed point) of some unknown operator $F$ represented by the form of an expectation, i.e., $F(x) = \Eset_{\pi}[F(x,\xi)]$, where $\xi$ is some random variable with a distribution $\pi$. Specifically, this method seeks a point $x^{\star}$ such that $F(x^{\star}) = 0$ based on the noisy observations $F(x;\xi)$. The estimate $x$ of $x^{\star}$ is iteratively updated by moving along the direction of $F(x;\xi)$ scaled by some step size. Through a careful choice of this step size, the ``noise'' induced by the random samples $\xi$ can be averaged out across iterations, and the algorithm converges to $x^*$. Both asymptotic convergence and finite-sample complexity of SA are very well-studied under general conditions of the operators and noise models; see for example \cite{borkar2008,benveniste2012adaptive,LanBook2020, BottouCN2018,Bhandari2018_FiniteTD,Karimi_colt2019, SrikantY2019_FiniteTD,HuS2019,Chen_MC_LinearSA_2020,ChenZDMC2019}. One can apply SA to solve problem \eqref{prob:FG}, i.e., by considering $W = [F;G]$ as a single operator. This approach requires that $W$ has to satisfy a similar condition as used in the existing literature of SA. On the other hand, in this paper we consider a more general condition of $F,G$ that is not applicable to the existing results of SA (see the discussion in Remark 2 for more details).       

\noindent\textbf{Two-Time-Scale Stochastic Approximation} is proposed to find the solution $(x^{\star},y^{\star})$ of \eqref{prob:FG} under weaker conditions of the operators $F,G$ when SA is not applicable. Indeed, in many applications presented in Section \ref{sec:motivating_applications}, two-time-scale SA is observed to perform better than its single-time-scale counterpart, either more stable or converge faster. Unfortunately, this advantage comes at a price: the convergence analysis of the two-time-scale SA in \eqref{alg:xy}, especially its finite-sample complexity, is more challenging to study than SA. To date, the convergence of this method, both asymptotic and finite-time convergence rate, is only well-studied in the linear setting; see for example in \cite{KondaT2004, DalalTSM2018, DoanR2019,GuptaSY2019_twoscale,Doan_two_time_SA2019,Kaledin_two_time_SA2020,Dalal_Szorenyi_Thoppe_2020,haque2023tight}. In particular, the asymptotic rate of \eqref{alg:xy} is probably first studied in \cite{KondaT2004}, where the authors show an $\Ocal(1/k)$ convergence in distribution via central limit theorem. Following this work, several results on the convergence properties of \eqref{alg:xy} have been studied, e.g., convergence with high probability in \cite{Dalal_Szorenyi_Thoppe_2020}, and convergence in mean-square errors with constant step sizes in \cite{GuptaSY2019_twoscale} and time-varying step sizes in \cite{8919880,Doan_two_time_SA2019}. The works in \cite{8919880, Doan_two_time_SA2019} only achieve a rate $\Ocal(1/k^{2/3})$, which is considered as suboptimal. The convergence rate of \eqref{alg:xy} in the linear setting is later shown to be $\Ocal(1/k)$ in \cite{Kaledin_two_time_SA2020}, utilizing the setting in \cite{KondaT2004}. Recently, this result is further refined with tighter constants in the upper bounds of the convergence rates in \cite{haque2023tight}.

Unlike the linear setting, theoretical results on the convergence rates of \eqref{alg:xy} in the general nonlinear setting are less understood. In \cite{MokkademP2006}, the authors provide an asymptotic convergence rate of \eqref{alg:xy} in distribution via central limit theorem when the noise is i.i.d. This results is recently generalized to the Markov setting, i.e., noise is dependent and samples are biased, in \cite{hu2024central}. On the other hand, the first finite-time complexity of \eqref{alg:xy} in the nonlinear setting is studied in \cite{Doan_two_time_SA2020, zeng2021two}, where the author shows a rate of $\Ocal(1/k^{2/3})$ in the mean-squared regime. This result is shown by leveraging the singular perturbation technique in control theory \cite{Kokotovic_SP1999} that is first proposed to study the finite-time complexity of linear two-time-scale SA in \cite{GuptaSY2019_twoscale}. This result is recently refined in \cite{han2024finite}, where the authors obtain an $\Ocal(1/k)$ convergence rate. However, the analysis in \cite{han2024finite} requires a rather restrictive condition on the high order of the smoothness of the operators, which may not hold for any nonlinear operator (see the discussion in Remark \ref{remark:thm}). The focus of this work is to develop a new variant of two-time-scale SA that gives an $\Ocal(1/k)$ convergence rate without using an assumption on the high order of the smoothness of the operators.



\subsection{Motivating Applications}\label{sec:motivating_applications}
In this section, we present a few existing applications of two-time-scale SA in \eqref{alg:xy}, which serve as the main motivation of this paper. Our improved convergence rate in Theorem \ref{thm:rate} below suggests that one can apply the development in this paper to develop new algorithms for solving problems in these applications.

\subsubsection{Reinforcement Learning.} Two-time-scale SA has been used extensively to model  reinforcement learning methods, for example, gradient temporal difference (TD) learning  and actor-critic methods \cite{Sutton2009a,Sutton2009b,Maeietal2009,KondaT2003,xu_actor_critic2020,wu_actor_critic2020,Hong_actor_critic2020,Khodadadian_actorcritic_2021,9992837}. For example, the gradient TD learning for solving the policy evaluation problem under nonlinear function approximations studied in \cite{Maeietal2009} can be viewed as a variant of \eqref{alg:xy}. In this problem, we want to estimate the cumulative rewards $V$ of a stationary policy using function approximations $V_{y}$, that is, our goal is to find $y$ so that $V_{y}$ is as close as possible to the true value $V$. Here, $V_{y}$ can be represented by a neural network where $y$ is the weight vector of the network. Let $\zeta$ be the environmental sate, $\gamma$ be the discount factor, $\phi(\zeta) = \nabla V_{y}(\zeta)$ be the feature vector of state $\zeta$, and $r$ be the reward return by the environment. Given a sequence of samples $\{\zeta_{k},r_{k}\}$, one version of  {GTD} are given as
\begin{align*}
x_{k+1} &= x_{k} + \alpha_{k}( \delta_{k} - \phi(\zeta_{k})^{T}x_{k})\phi(\zeta_{k})\\
y_{k+1} &= y_{k} + \beta_{k}\Big[\left(\phi(\zeta_{k})-\gamma\phi(\zeta_{k+1})\right)\phi(\zeta_{k})^Tx_{k} - h _{k}\Big],
\end{align*}
where $\delta_{k}$ and $h_{k}$ are defined as  
\begin{align*}
\delta_{k} &= r_{k} + \gamma V_{y_{k}}(\zeta_{k+1}) - V_{y_{k}}(\zeta_{k})\\ 
h_{k} &= (\delta_{k} - \phi(\zeta_{k})^Tx_{k})\nabla^2V_{y_{k}}(\zeta_{k})x_{k},    
\end{align*}
which is clearly a variant of \eqref{alg:xy} under some proper choice of $F$ and $G$. It has been observed that the {GTD} method is more stable and performs better when compared to the single-time-scale counterpart (TD learning) under off-policy learning and nonlinear function approximations \cite{SBbook2018}. 

Another reinforcement learning method, namely, the online actor-critic algorithm \cite{KondaT2003}, to find the underlying optimal policy can also be viewed as a variant of \eqref{alg:xy}. In this algorithm, there is a critic used to estimate the state-action value function of the current policy. The estimate of the critic is then used by an actor to update the current policy toward the optimal policy. It is known that the update of the critic is implemented at a faster time scale (e.g., using a bigger step size) than the update of the actor (e.g., using a smaller step size), resulting a two-time-scale learning algorithm.         

While the existing gradient TD learning and online actor-critic methods can be formulated as different variants of \eqref{alg:xy}, the proposed method in this paper suggests that it will be better to use the Ruppert-Polyak averaging steps on the samples to reduce the impacts of noise before updating the main iterates in these two methods. In Section \ref{sec:simulations} we provide numerical experiments to illustrate this observation.

\subsubsection{Game Theory.} Two-time-scale SA can also be used to represent different learning methods in game theory, in particular, in zero-sum settings. For example, consider a two-player Markov game characterized by $\Mcal=(\Scal,\Acal,\Bcal,\Pcal,\gamma,r)$. Here, $\Scal$ is the finite state space, $\Acal$ and $\Bcal$ are the finite action spaces of the two players, $\gamma\in(0,1)$ is the discount factor, and $r:\Scal\times\Acal\times\Bcal\rightarrow[0,1]$ is the reward function. Let $\Delta_{\Fcal}$ denote the probability simplex over a set $\Fcal$, and $\Pcal:\Scal\times\Acal\times\Bcal\rightarrow\Delta_{\Scal}$ be the transition probability kernel, with $\Pcal(s'\mid s,a,b)$ specifying the probability of the game transitioning from state $s$ to $s'$ when the first player selects action $a\in\Acal$ and the second player selects $b\in\Bcal$.
The policies of the two players are denoted by $\pi\in\Delta_{\Acal}^{\Scal}$ and $\phi\in\Delta_{\Bcal}^{\Scal}$, with $\pi(a\mid s)$, $\phi(b\mid s)$ denoting the probability of selecting action $a$, $b$ in state $s$ according to $\pi$, $\phi$.
Given a policy pair $(\pi,\phi)$, we measure its performance in state $s\in\Scal$ by the value function
\begin{align*}
    V^{\pi, \phi}(s)=\mathbb{E}_{\{a_k,b_k,s_{k+1}\}}\Big[\sum\nolimits_{k=0}^{\infty} \gamma^k r\left(s_k, a_k, b_k\right) \mid s_0=s\Big].
\end{align*}
Under a fixed initial distribution $\rho\in\Delta_{\Scal}$, we define by $J(\pi, \phi)\triangleq\mathbb{E}_{s_0\sim\rho}[V^{\pi, \phi}(s_0)]$ the discounted cumulative reward under $(\pi,\phi)$. The goal is to find a Nash equilibrium $(\pi^{\star},\phi^{\star})$, which can be shown to be the solution of the following minimax optimization problem
\begin{align*}
\max_{\pi\in\Delta_{\Acal}^{\Scal}}\min_{\phi\in\Delta_{\Bcal}^{\Scal}}J(\pi, \phi)=\min_{\phi\in\Delta_{\Bcal}^{\Scal}}\max_{\pi\in\Delta_{\Acal}^{\Scal}}J(\pi, \phi)=J(\pi^{\star},\phi^{\star}).
\end{align*}
When the transition probability kernel $\Pcal$ is known, one can apply dynamic programming approach to solve this problem. On the other hand, when $\Pcal$ is unknown one can use reinforcement learning methods to find the Nash equilibrium. For example, we can apply the decentralized Q-learning method recently studied in \cite{sayin2021decentralized}, which is a variant of two-time-scale SA. This method maintains two variables, namely, local value functions and local q-function for each player, that are updated at different time scales by using different step sizes. In particular, let $v_{1}^{\pi}(s), q_{1}^{\pi}(s,a)$ be the local value function and state-action value function at Player $1$, respectively. Similar notation is defined for Player $2$. Player $1$ will update its local functions as (similar to Player $2$)
\begin{align*}
q_{i}^{k+1}(s_{k},a_{k}) &= q_{i}^{k+1}(s_{k},a_{k}) + \alpha_{k}\left(r(s_{k},a_{k},b_{k}) + \gamma v_{i}^{k}(s_{k}') - q_{i}^{k+1}(s_{k},a_{k})\right),\\
v_{i}^{k+1}(s_{k}) &= v_{i}^{k}(s_{k}) + \beta_{k}\big(\big[\bar{\pi}_{i}^{k}\big]^T q_{i}^{k+1}(s_{k},\cdot) - v_{i}^{k},(s_{k})\big),
\end{align*}
where $\beta_{k} \leq \alpha_{k}$ and $\bar{\pi}_{i}^{k}$ is the best response policy w.r.t $q_{i}^{k+1}$. While asymptotic convergence of this method has been studied in \cite{sayin2021decentralized} based on the ODE method, its finite-time complexity is completely unknown. Thus, an interesting question is to study whether one can apply the technique studied in this paper to derive the convergence rate of the decentralized Q-learning method. Another question is to see whether one can improve the existing convergence rates of policy gradient descent-ascent algorithm for zero-sum Markov games studied in \cite{zeng2022regularized}, where the authors show that this algorithm converges at a rate $\Ocal(1/k^{1/3})$. Although this result is not comparable to the one studied in Theorem \ref{thm:rate}, the objective function $J(\pi,\phi)$ is nonconvex and nonconcave, therefore, one cannot immediately apply our analysis to the zero-sum Markov game setting.

\subsubsection{Minimax Optimization.} We consider the following minimax optimization problems 
\begin{align*}
\min_{x\in\Rset^{m}}\max_{y\in\Rset^{n}}f(x,y),
\end{align*}
where $f:\Rset^{m}\times \Rset^{n}\rightarrow\Rset$ is a (non)convex function w.r.t $x$ for a fixed $y$ and (non)concave w.r.t $y$ for a fixed $x$. The minimax problem has broad applications in different areas including game theory \cite{Basarbook1998},   training generative adversarial networks (GANs) \cite{Goodfellow_GAN_2020,Mescheder2017}, adversarial and robust machine learning \cite{KurakinGB17,Qian_Zhu_Tang_Jin_Sun_Li_2019}, resource allocation over networks \cite{6450111}, and distributed optimization \cite{Lan2020_DecentralizedOpt,9085431}; to name just a few. For solving this problem, (alternating) gradient descent-ascent is the most popular method used in the existing literature; see for example  \cite{lin20a,9070155,Yang_NEURIPS2020,Xu2020AUS,Zhang_NEURIPS2020,doan2022convergence,Cherukuri_2017} and the references therein. This method iteratively updates the estimates $x_{k},y_{k}$ of the desired solution as  
\begin{align}
\begin{aligned}
    x_{k+1} &= x_{k} - \alpha_{k}\nabla_{x} f (x_{k},y_{k};\xi_{k}),\\ 
    y_{k+1} &= y_{k} - \beta_{k}\nabla_{y} f (x_{k},y_{k};\psi_{k}),
\end{aligned}
\end{align}
where $\alpha_{k},\beta_{k}$ are chosen differently to guarantee the convergence of this method. It is obvious that gradient descent ascent is a variant of the two-time-scale SA. 

\subsubsection{Other Applications.} Finally, two-time-scale SA has been used in distributed optimization to address the issues of communication constraints \cite{DoanBS2017,DoanMR2018b,8619539,9683681} and in distributed control to handle clustered networks \cite{Romeres13,JChow85,Biyik08,Boker16,pham2021distributed,dutta2022convergence,9992900}. For example, the best convergence rate of distributed consensus-based policy gradient methods under random quantization is $\Ocal(1/\sqrt{k})$ when the objective function is strongly convex \cite{9683681}, a variant of strong monotonicity in Assumption \ref{assump:sm}. Thus, it will be interesting to see whether we can apply the proposed technique in this paper to improve the convergence rate of distributed optimization method under random quantization.

\section{Fast Two-Time-Scale Stochastic Approximation}
The main difficulty in analyzing the convergence of the classic two-time-scale SA in \eqref{alg:xy} is to handle the coupling between the two iterates, i.e., the update of $x$ depends on $y$ and vice versa. A common approach to address this problem is to choose proper step sizes $\beta_{k}\ll \alpha_{k}$, e.g., $\beta_{k} \sim 1/k \leq \alpha_{k}\sim 1/k^{2/3}$. Then one can show that the update of $x$ is asymptotically independent on $y$ and vice versa. However, under the impact of stochastic noise due to the sampling of operators, such a choice of step sizes yields a suboptimal finite-time convergence rate of two-time-scale SA, e.g., $\Ocal(1/k^{2/3})$.  In this paper, we propose an alternative approach that will help to decouple the impact of noise on the update of the iterates. In particular, we consider the following variant of two-time-scale SA in
\eqref{alg:xy}
\begin{align}
&\left\{
\begin{array}{l}
f_{k+1} = (1-\lambda_{k})f_{k} + \lambda_{k} (F(x_{k},y_{k}) + \xi_{k})\\
x_{k+1} = x_{k} - \alpha_{k}f_{k},
\end{array}
\right.\label{alg:x-accelerated}\\
&\hspace{-1.1cm} \text{and}\notag\\ 
&\left\{
\begin{array}{l}
g_{k+1} = (1-\gamma_{k})g_{k} + \gamma_{k} (G(x_{k},y_{k}) + \psi_{k})\\
y_{k+1} = y_{k} - \beta_{k}g_{k}.
\end{array}
\right.\label{alg:y-accelerated}
\end{align}
Here, $f_{k}$ and $g_{k}$ are used to estimate the time-weighted average of the operators by using the two nonnegative step sizes $\gamma_{k}$ and $\lambda_{k}$. These estimates are then used to update the iterates $x,y$. When $\lambda_{k} = \gamma_{k} = 1$, the  updates in  \eqref{alg:x-accelerated} and \eqref{alg:y-accelerated} reduce to the ones in \eqref{alg:xy}. On the other hand, when $\lambda_{k},\gamma_{k} < 1$ one can view the updates of $f_{k}$ and $g_{k}$ as a variant of the classic Ruppert-Polyak averaging technique \cite{ruppert1988efficient, polyak1992acceleration}. Ruppert-Polyak averaging technique is often applied to estimate the time-weighted average values of the iterates in optimization and SA literature. The average variable often achieves a better performance than the one without using averaging; see for example \cite{LanBook2020,haque2023tight,chen2021accelerating} and the references therein. In the proposed method in  \eqref{alg:x-accelerated} and \eqref{alg:y-accelerated}, we utilize the Ruppert-Polyak averaging technique on the operators, which will help to decompose the coupling between sampling noise and the iterate updates. Indeed, we will show that by choosing $\lambda_{k},\gamma_{k} < 1$, one can separate the impact of the noise on the updates of the iterates, therefore, achieving an optimal rate $\Ocal(1/k)$ under a proper choice of step sizes $\alpha_{k},\beta_{k}$.  

\begin{remark}
In our approach, we use the Ruppert-Polyak averaging steps on the samples of the operators, which is different to the existing works where these steps are often applied to the main iterates. The benefit of using this averaging technique to the iterates has been observed in the literature. On the other hand, we are not aware of any prior work that uses this technique to improve the estimate of the operators in parallel with the iterate updates. An interesting question is to study the performance of these two types of averaging steps when they are implemented in parallel.         
\end{remark}


\subsection{Notation and Main Assumptions}\label{sec:assumptions}
In this section, we present the main assumptions which we will be used to study the finite-time convergence rate of \eqref{alg:x-accelerated} and \eqref{alg:y-accelerated}. These assumptions are fairly standard and have been used in prior works under different variants \cite{KondaT2004, DalalTSM2018, DoanR2019,GuptaSY2019_twoscale,Doan_two_time_SA2019,Kaledin_two_time_SA2020,MokkademP2006}.  Our first assumption is on the smoothness of the operators.
\begin{assump}\label{assump:smooth}
Given $y\in\Rset^{d}$ there exists an operator $H:\Rset^{d}\rightarrow\Rset^{d}$ such that $x = H(y)$ satisfies
\begin{align*}
F(H(y),y) = 0.    
\end{align*}
In addition, we assume that  the operators $F,G,H$ are Lipschitz continuous with a positive constant $L$, i.e., the following conditions satisfy for all $x_{1}, x_{2}, y_{1}, y_{2} \in\Rset^{d}$
\begin{align}
\begin{aligned}
\hspace{-0.2cm}\|H(y_{1}) - H(y_{2})\| &\leq L\|y_{1}-y_{2}\|,\\
\hspace{-0.2cm}\|F(x_{1},y_1) - F(x_{2},y_2)\| &\leq L(\|x_{1}-x_{2}\| + \|y_{1} - y_{2}\|),\\
\hspace{-0.2cm}\|G(x_{1},y_{1}) - G(x_{2},y_{2})\| &\leq L\left(\|x_{1} - x_{2}\| + \|y_{1} - y_{2}\| \right).   
\end{aligned}
\label{assump:smooth:ineq}
\end{align}
\end{assump}
\noindent Second, we will study the setting when $F,G$ are strongly monotone, which implies that $(x^{\star},y^{\star})$ is unique.
\begin{assump}\label{assump:sm}
Given $y$, $F$ is strongly monotone w.r.t $x$, i.e., there exists a constant $\mu_{F} > 0$ such that for all $x,z\in\Rset^{d}$
\begin{align}
\left\langle x - z, F(x,y) - F(z,y) \right\rangle \geq \mu_{F} \|x-z\|^2. \label{assump:sm:F:ineq}  
\end{align}
Moreover, $G$ is $1$-point strongly monotone w.r.t $y^{\star}$, i.e., there exists a constant $\mu_{G} \in (0,\mu_{F}]$ such that for all $y\in\Rset^{d}$
\begin{align}
\left\langle y - y^{\star}, G(H(y),y) \right\rangle \geq \mu_{G} \|y - y^{\star}\|^2. \label{assump:sm:G:ineq}
\end{align}
\end{assump}

\begin{remark}\label{remark:assumption}
First, Assumption \ref{assump:smooth} is often referred to as the global Lipschitz continuity. Our analysis in this paper can be easily extended to the local Lipschitz condition as studied in \cite{hu2024central} (with extra constant terms in our theoretical bounds), i.e., $H$ (similarly, $F,G$) satisfies
\begin{align*}
\|H(y)\| \leq L(1 + \|y\|)\quad \forall y\in\Rset^{d}.     
\end{align*}
Second, it is also worth to point out that in general Assumption \ref{assump:sm} does not imply that the operator $W = [F;G]$ is strongly monotone. Indeed, if $W$ is strongly monotone, one can apply the single-time-scale SA to solve \eqref{prob:FG}, e.g., by considering SA to update $z = [x,y]$. This statement can be verified in the linear setting, where  
$F$ and $G$ are linear operators 
\begin{align*}
F(x,y) &= \Abf_{11}x + \Abf_{12}y - b_{1}\\
G(x,y) &= \Abf_{12}x + \Abf_{22}y - b_{2}.
\end{align*}
In this case, Assumption \ref{assump:sm} implies that $\Abf_{11}$ and $\Delta = \Abf_{22}-\Abf_{21}\Abf_{11}^{-1}\Abf_{12}$ are negative definite (but not necessarily symmetric), and $H(y) = \Abf_{11}^{-1}(b_{1} - \Abf_{12}y)$. It is straightforward to see that $(x^*,y^*)$ is unique and satisfies
\begin{align*}
x^* &= \Abf_{11}^{-1}(b_{1}-\Abf_{12}y^*),\\
y^* &= \Delta^{-1}\left(b_2 - \Abf_{21}\Abf_{11}^{-1}b_{1}\right).
\end{align*} 
However, the matrix $\Abf$ defined below is not negative definite
\begin{align*}
\Abf = \left[\begin{array}{cc}
\Abf_{11}     &  \Abf_{12}\\
\Abf_{12}     &  \Abf_{22}
\end{array} \right].
\end{align*}
To see this, we consider the following example 
\begin{align*}
\Abf_{11} &= \left[\begin{array}{cc}
-3.5     &  -8.5\\
1     & 0 
\end{array}\right],\quad \Abf_{12} = \left[\begin{array}{cc}
2     &  0\\
0     & 2 
\end{array}\right],\\
\Abf_{21} &= \left[\begin{array}{cc}
5     &  0\\
0     & 5 
\end{array}\right],\quad \Abf_{22} = \left[\begin{array}{cc}
-10     &  -1.6\\
20     & -1 
\end{array}\right],
\end{align*}
Then the eigenvalues of $\Abf_{11}$, $\Delta$, and $\Abf$ are
\begin{align*}
\text{eig}(\Abf_{11}) &=
-1.75 \pm 2.33i
,\quad \text{eig}(\Delta) = 
-3.44 \pm 14.24i\\
\text{eig}(\Abf) &= \{-8.53 \pm 4.51i,\;
1.28 \pm 4.24i\},
\end{align*}
which shows that $\Abf_{11}$ and $\Delta$ are negative but $\Abf$ is not.    
\end{remark}
Finally, we consider i.i.d noise models, that is, $\xi_{k}$ and $\psi_{k}$ are Martingale difference. We denote by $\Qcal_{k}$ the filtration containing all the history generated by \eqref{alg:xy} up to time $k$, i.e.,
\begin{equation*}
\Qcal_{k} = \{x_{0},y_{0},\xi_{0},\psi_{0},\xi_{1},\psi_{1},\ldots,\xi_{k},\psi_{k}\}.
\end{equation*}
\begin{assump}\label{assump:noise}
The random variables $\xi_{k}$ and $\psi_{k}$, for all $k\geq0$, are independent of each other and across time, with zero mean and common variances given as follows
\begin{equation*}
\Eset[\xi_{k}^T\xi_{k}\,|\,\Qcal_{k-1}] = \Gamma_{11},\quad
\Eset[\psi_{k}^T\psi_{k}\,|\,\Qcal_{k-1}] = \Gamma_{22}.
\end{equation*}
\end{assump}
\begin{remark}
    We note that our analysis in this paper can be extended to cover the Markov setting where $\xi_{k}$ and $\psi_{k}$ are generated by Markov processes. In this case, $\xi_{k}$ and $\psi_{k}$ are dependent and the samples are biased. In this setting, if the Markov processes are ergodic, i.e., they have geometric mixing time, one can apply the technique proposed in \cite{SrikantY2019_FiniteTD} to handle the dependent and biased samples. This will result in an extra factor $\Ocal(\log(k))$ in our theoretic bounds, where $k$ is the number of iterations.   
\end{remark}


\section{Main Results}\label{sec:results}
In this section, we present in details the main results of this paper, that is, we provide a finite-time analysis for the convergence rates of the two-time-scale stochastic approximation in \eqref{alg:x-accelerated} and \eqref{alg:y-accelerated}. Under assumptions presented in Section \ref{sec:assumptions}, we show that the mean square errors of these iterates, $\Eset[\|x_{k}-x^*\|^2]$ and $\Eset[\|y_{k}-x^*\|^2]$, converge to zero at a rate $\Ocal(1/k)$. To show this result, we introduce the following four residual variables, 
\begin{align}
\begin{aligned}
\Delta f_{k} = f_{k} - F(x_{k},y_{k}) \quad \text{and}\quad\xhat_{k} &= x_{k} - H(y_{k}),\\   
\Delta g_{k} = g_{k} - G(x_{k},y_{k}) \quad \text{and}\quad \yhat_{k} &= y_{k} - y^{\star}.
\end{aligned}    \label{alg:xyhat}
\end{align}
One can rewrite the updates of $x,y$ in  \eqref{alg:x-accelerated} and \eqref{alg:y-accelerated} as
\begin{align*}
x_{k+1} &= x_{k} - \alpha_{k}F(x_{k},y_{k}) - \alpha_{k}\Delta f_{k}\\ 
y_{k+1} &= y_k - \beta_{k}G(x_{k},y_{k}) - \beta_{k}\Delta g_{k},     
\end{align*}
If $(\Delta f_{k},\Delta g_{k})$ goes to zero then the impact of sampling noise on $x,y$ will be negligible. In addition, if  $(\xhat_{k},\yhat_{k})$ go to zero then $(x_{k},y_{k})\rightarrow(x^{\star},y^{\star})$. Thus, to establish the convergence of $(x_{k},y_{k})$ to $(x^{\star},y^{\star})$ one can instead study the convergence of $(\Delta f_{k},\Delta g_{k})$ and $(\xhat_{k},\yhat_{k})$ to zero. Indeed, our main result presented in Theorem \ref{thm:rate} below is to study an upper bound for the convergence rate of these quantities. We will do this by considering the following choice of step sizes
\begin{align}
\begin{aligned}
\lambda_{k} &= \frac{C_{\lambda}}{(k+h+1)},\quad \gamma_{k} = \frac{C_{\gamma}}{(k+h+1)}\\ 
\alpha_{k} &= \frac{C_{\alpha}}{k+h+1},\quad \beta_{k} = \frac{C_{\beta}}{k+h+1}, 
\end{aligned}\label{step-sizes}
\end{align}
where $C_{\lambda},C_{\gamma},C_{\alpha}, C_{\beta}, h$ are some positive constants, which will be specified in Theorem \ref{thm:rate}. We will consider the setting when $\beta_{k}\leq \alpha_{k}$, i.e., the iterate $x_{k}$ is implemented at a faster time scale than $y_{k}$. To the end of this paper, we will assume that Eq. \eqref{step-sizes} always holds.


\subsection{Preliminaries}\label{sec:preliminaries}
We first present some preliminary results, which will help to facilitate the analysis of Theorem \ref{thm:rate}. In particular, the following lemmas are to provide some important bounds for the size of the variables generated by the updates in \eqref{alg:x-accelerated} and \eqref{alg:y-accelerated}. For an ease of exposition, we present the analysis of these lemmas in the appendix.   

\begin{lem}\label{lem:inequalities}
We have
\begin{align}
\begin{aligned}
\hspace{-0.2cm}\|g_{k}\| &\leq \|\Delta g_{k}\| + L\|\xhat_{k}\| + L(1+L)\|\yhat_{k}\|,\\
\hspace{-0.2cm}\|f_{k}\| &\leq \|\Delta f_{k}\| + L\|\xhat_{k}\|,\\
\hspace{-0.2cm}\|x_{k+1}-x_{k}\| &\leq \alpha_{k}(\|\Delta f_{k}\| + L\|\xhat_{k}\|),\\ 
\hspace{-0.2cm}\|y_{k+1}-y_{k}\| &\leq  \beta_{k} (\|\Delta g_{k}\| + L\|\xhat_{k}\| + L(1+L)\|\yhat_{k}\|).
\end{aligned}
\label{lem:inequalities:Ineq} 
\end{align}
\end{lem}


\begin{lem}\label{lem:Delta-f}
Let $\lambda_{k}$ be chosen as $\lambda_{k}\leq 1/4$.
Then we have
\begin{align}
\|\Delta f_{k+1}\|^2 &\leq  (1-\lambda_{k})\|\Delta f_{k}\|^2  + 2\lambda_{k}^2\|\epsilon_{k}\|^2 - \Big(\frac{1}{4} - \frac{8L^2\alpha_{k}^2}{\lambda_{k}^2} \Big)\lambda_{k} \|\Delta f_{k}\|^2+ \frac{20L^4\alpha_{k}^2}{\lambda_{k}}\|\xhat_{k}\|^2 \notag\\
&\quad   + \frac{12L^4(1+L)^2\beta_{k}^2}{\lambda_{k}}\|\yhat_{k}\|^2  + 2(1-\lambda_{k})\lambda_{k}\epsilon_{k}^T\Delta f_{k} + \frac{12L^2\beta_{k}^2}{\lambda_{k}}\|\Delta g_{k}\|^2
.\label{lem:Delta-f:ineq}
\end{align}
\end{lem}

\begin{lem}\label{lem:Delta-g}
Let $\gamma_{k}$ be chosen as $\gamma_{k} \leq 1/4$. 
Then we have
\begin{align}
\|\Delta g_{k+1}\|^2
&\leq (1-\gamma_{k})\|\Delta g_{k}\|^2 - \Big(\frac{1}{4} - \frac{12L^2\beta_{k}^2}{\gamma_{k}^2}\Big)\gamma_{k} \|\Delta g_{k}\|^2 + 2\gamma_{k}^2\|\psi_{k}\|^2 + 2(1-\gamma_{k})\gamma_{k}\psi_{k}^T\Delta g_{k}  \notag\\
&\quad  + \frac{8L^2\alpha_{k}^2}{\gamma_{k}}\|\Delta f_{k}\|^2    + \frac{20L^4\alpha_{k}^2}{\gamma_{k}}\|\xhat_{k}\|^2  + \frac{12L^4(1+L)^2\beta_{k}^2}{\gamma_{k}}\|\yhat_{k}\|^2.\label{lem:Delta-g:ineq}
\end{align}
\end{lem}

\begin{lem}\label{lem:xhat}
Let $\alpha_{k}$ be chosen as 
\begin{align}
\alpha_{k} \leq \frac{\mu_{F}}{L^2}\cdot    \label{lem:xhat:stepsizes}
\end{align}
Then we have
\begin{align}
\|\xhat_{k+1}\|^2 &\leq  \left(1 -\mu_{F}\alpha_{k}\right)\|\xhat_{k}\|^2 + L^2\big(4L^2 + 1\big)\alpha_{k}^2\|\xhat_{k}\|^2 - \left[\frac{\mu_{F}\alpha_{k}}{2} - \frac{4L^3\big(1+L(1+L)^2\big)\beta_{k}}{\mu_{G}}\right]\|\xhat_{k}\|^2\notag\\
&\quad + \frac{\mu_{G}\beta_{k}}{4}\|\yhat_{k}\|^2 + 4L^4(1+L)^{2} \alpha_{k}\beta_{k}\|\yhat_{k}\|^2 + \left[\frac{2\alpha_{k}}{\mu_{F}} + 4\alpha_{k}^2\right]\|\Delta f_{k}\|^2\notag\\ 
&\quad + \left[\beta_{k} + 4L^2\alpha_{k}\beta_{k} \right]\|\Delta g_{k}\|^2.\label{lem:xhat:Ineq}
\end{align}
\end{lem}

\begin{lem}\label{lem:yhat}
Let $\beta_{k}$ be chosen as 
\begin{align}
\beta_{k} \leq \frac{\mu_{G}}{4L^2(1+L)^2}\cdot    \label{lem:yhat:stepsizes}
\end{align}
Then we have
\begin{align}
\|\yhat_{k+1}\|^2&\leq (1 - \mu_{G}\beta_{k})\|\yhat_{k}\|^2 - \frac{\mu_{G}\beta_{k}}{2}\|\yhat_{k}\|^2  + \left(\frac{8\beta_{k}}{\mu_{G}} + 2\beta_{k}^2\right) \Big( L^2\|\xhat_{k}\|^2 + \|\Delta g_{k}\|^2\Big). \label{lem:yhat:Ineq}
\end{align}
\end{lem} 

\subsection{Finite-Time Convergence Rates}
Leveraging the preliminary results in the previous section, we are ready to present the main result of this paper. In Theorem \ref{thm:rate} below, we will show that the iterates generated by the proposed method in \eqref{alg:x-accelerated} and \eqref{alg:y-accelerated} will converge to the desired solution $(x^{\star},y^{\star})$ in expectation. Our analysis will be based on the following candidate Lyapunov function  
\begin{align}
V_{k} = \|\Delta f_k\|^2 + \|\Delta g_{k}\|^2 + \|\xhat_{k}\|^2 + \|\yhat_{k}\|^2. \label{Lyapunov}   
\end{align}
We will show that each term on the right-hand side of \eqref{Lyapunov} will converge to zero at the same rate. Therefore, we do not need to use the coupling Lyapunov function as in \cite{GuptaSY2019_twoscale, Doan_two_time_SA2020}. 
\begin{thm}\label{thm:rate}
Let $\{f_{k}, g_{k}, x_{k},y_{k}\}$ be generated by \eqref{alg:x-accelerated} and \eqref{alg:y-accelerated}, and $\{\lambda_{k},\gamma_{k},\alpha_{k},\beta_{k}\}$ defined in \eqref{step-sizes} with $\lambda_{k} = \gamma_{k} \leq 1/4$ and
\begin{align}
\begin{aligned}
&\frac{\beta_{k}}{\alpha_{k}} \leq \frac{\mu_{F}\mu_{G}}{32L^2(1+L)^4},\\ 
&\alpha_{k} \leq \min \left\{\frac{\mu_{F}}{32L^2(L+1)^2};\; \frac{\mu_{G}}{L^4(1+L)^2}\right\},\\
& \frac{\alpha_{k}}{\lambda_{k}}\leq \min\left\{\frac{1}{24L};\; \frac{\mu_{G}}{144(1+\mu_{G})};\;\frac{\mu_{G}}{192L^4(1+L)^2}\right\}.
\end{aligned}
\label{thm:rate:stepsizes}
\end{align}
Then we obtain
\begin{align}
\Eset[V_{k+1}] &\leq (1-\mu_{G}\beta_{k})\Eset[V_{k}] + 2\lambda_{k}^2\Eset\Big[\|\epsilon_{k}\| + \|\psi_{k}\|^2\Big]\cdot \label{thm:rate:Ineq}   
\end{align}
\end{thm}
\begin{proof}
It is straightforward to verify that under the conditions in \eqref{thm:rate:stepsizes}, $\{\lambda_{k},\gamma_{k},\alpha_{k},\beta_{k}\}$ satisfy the conditions in Lemmas \ref{lem:Delta-f}--\ref{lem:yhat} and $\beta_{k}\leq \alpha_{k}$. In addition, by Assumption \ref{assump:noise} we have   $\Eset[\epsilon_{k}\mid \Qcal_{k}] = \Eset[\psi_{k}\mid \Qcal_{k}] = 0.$   Thus, taking the conditional expectation w.r.t $\Qcal_{k}$ on both sides of Eqs.\ \eqref{lem:Delta-f:ineq}, \eqref{lem:Delta-g:ineq}, \eqref{lem:xhat:Ineq}, and \eqref{lem:yhat:Ineq} we obtain
\begin{align}
\Eset[V_{k+1}\mid \Qcal_{k}]&= \Eset[\|\Delta f_{k+1}\|^2 + \|\Delta g_{k+1}\|^2 + \|\xhat_{k+1}\|^2 + \|\yhat_{k+1}\|^2\mid \Qcal_{k}]\notag\\
&= (1-\lambda_{k})\|\Delta f_{k}\|^2 + (1-\gamma_{k})\|\Delta g_{k}\|^2 + \left(1 -\mu_{F}\alpha_{k}\right)\|\xhat_{k}\|^2 + (1 - \mu_{G}\beta_{k})\|\yhat_{k}\|^2  \notag\\
&\quad + 2\lambda_{k}^2\Eset\left[\|\epsilon_{k}\|^2 + \|\psi_{k}\|^2\mid \Qcal_{k}\right]  - A_{1} - A_{2} - A_{3} - A_{4},
\label{thm:rate:Eq1}
\end{align}
where $A_{i}$, for $i=1,2,3,4,$ are defined as follows
\begin{align*}
A_{1} &= \frac{\lambda_{k}}{4} - \frac{8L^2\alpha_{k}^2}{\lambda_{k}} - \frac{8L^2\alpha_{k}^2}{\gamma_{k}} - \frac{2\alpha_{k}}{\mu_{F}} - 4\alpha_{k}^2.\\
A_{2} &= \frac{\gamma_{k}}{4} - \frac{12L^2\beta_{k}^2}{\gamma_{k}} -\frac{12L^2\beta_{k}^2}{\lambda_{k}} - \frac{(8+\mu_{G})\beta_{k}}{\mu_{G}} - 4L^2\alpha_{k}\beta_{k} - 2\beta_{k}^2.\\
A_{3} &= \frac{\mu_{F}\alpha_{k}}{2} - \frac{4L^3\big(1+L(1+L)^2\big)\beta_{k}}{\mu_{G}} - \frac{8L^2\beta_{k}}{\mu_{G}} - \frac{20L^4\alpha_{k}^2}{\lambda_{k}} - \frac{20L^4\alpha_{k}^2}{\gamma_{k}} - L^2(4L^2+1)\alpha_{k}^2 - 2L^2\beta_{k}^2.\notag\\
A_{4} &= \frac{\mu_{G}\beta_{k}}{4} - \frac{12L^4(1+L)^2\beta_{k}^2}{\lambda_{k}}- \frac{12L^4(1+L)^2\beta_{k}^2}{\gamma_{k}}- 4L^4(1+L)^2\alpha_{k}\beta_{k}.
\end{align*}
Using the step-size conditions in \eqref{thm:rate:stepsizes} it is straightforward to verify that each $A_{i}$, for $i=1,2,3,4$, is nonnegative.  Therefore, since $\mu_{G}\leq \mu_{F}$, $\beta_{k}\leq \alpha_{k}$, and $\mu_{G}\alpha_{k}\leq \lambda_{k} = \gamma_{k}$, taking the expectation on both sides of \eqref{thm:rate:Eq1} gives \eqref{thm:rate:Ineq}, which conclude our proof.
\end{proof}
Using the result in Theorem \ref{thm:rate}, in the following lemma we present one choice of step sizes that gives a convergence rate $1/k$ of $V_{k}$ to zero in expectation.  

\begin{lem}
Let $\{f_{k}, g_{k}, x_{k},y_{k}\}$ be generated by \eqref{alg:x-accelerated} and \eqref{alg:y-accelerated}, and $\{\lambda_{k},\gamma_{k},\alpha_{k},\beta_{k}\}$ defined in \eqref{step-sizes} with 
\begin{align}
    \begin{aligned}
C_{\beta} &= \frac{2}{\mu_{G}},\quad C_{\alpha} = \frac{64L^2(1+L)^4}{\mu_{F}\mu_{G}^2},\\
C_{\lambda} &= C_{\gamma} \geq \frac{C_{\alpha}}{\min\left\{\frac{1}{24L};\; \frac{\mu_{G}}{144(1+\mu_{G})};\;\frac{\mu_{G}}{192L^4(1+L)^2}\right\}},\\
h &\geq  \frac{C_{\alpha}}{\min\left\{\frac{1}{4};\; \frac{\mu_{F}}{32L^2(1+L)^2};\;\frac{\mu_{G}}{L^4(1+L)^2}\right\}}\cdot
    \end{aligned}
    \label{lem:rate:stepsizes}
\end{align}
Then we have 
\begin{align}
\Eset[V_{k+1}]
\leq \frac{h^2\Eset[V_{0}]}{(k+h+1)^2} + \frac{C_{\lambda}^2(\Gamma_{11} + \Gamma_{22})}{k+h+1}\cdot    \label{lem:rate:Ineq}
\end{align}
\end{lem}

\begin{proof}
We first note that the conditions in \eqref{lem:rate:stepsizes} imply the ones in \eqref{thm:rate:stepsizes}. Thus, by using \eqref{thm:rate:Ineq} with $\beta_{k} = 2/(\mu_{G}(k+h+1))$ we obtain  
\begin{align*}
\Eset[V_{k+1}] &\leq (1-\mu_{G}\beta_{k})\Eset[V_{k}] + 2\lambda_{k}^2 \big(\Eset[\|\epsilon_{k}\|] + \Eset[\|\psi_{k}\|^2]\big)\notag\\
&= \left(1-\frac{2}{k+h+1}\right) \Eset[V_{k}] + \frac{C_{\lambda}^2(\Gamma_{11} + \Gamma_{22})}{(k+h+1)^2}\notag\\ 
& = \frac{k+h-1}{k+h+1} \Eset[V_{k}] + \frac{C_{\lambda}^2(\Gamma_{11} + \Gamma_{22})}{(k+h+1)^2},
\end{align*}
where the second inequality is due to Assumption \ref{assump:noise}. Multiplying both sides of the preceding relation by $(k+h+1)^2$ and since $h>1$ we obtain
\begin{align*}
(k+h+1)^2\Eset[V_{k+1}]&\leq (k+h+1)(k+h-1)\Eset[V_{k}] + C_{\lambda}^2(\Gamma_{11} + \Gamma_{22}) \notag\\   
&\leq (k+h)^2\Eset[V_{0}] + C_{\lambda}^2(\Gamma_{11} + \Gamma_{22})\notag\\
&\leq h^2\Eset[V_{0}] + C_{\lambda}^2(\Gamma_{11} + \Gamma_{22})(k+1),
\end{align*}
which when diving both sides by $(k+1+h)^2$ yields \eqref{thm:rate:Ineq}.

\begin{remark}\label{remark:thm}
In \cite{han2024finite}, the authors show that the updates in \eqref{alg:xy} can achieve an $\Ocal(1/k)$ convergence rate under a strong assumption on $H$, i.e., $H$ satisfies \eqref{assump:smooth:ineq} and 
\begin{align*}
\|H(y_{1}) - H(y_{2}) - \nabla H(y_{2})^T(y_1-y_{2})\| \leq L\|y_{1} - y_{2}\|^{\zeta},     
\end{align*}
where $\zeta \in [1.5, 2]$. This assumption basically implies that the impact of the noise in $y$ to $x$ update is negligible, i.e., on the high order of step sizes. Indeed, under this assumption the analysis in \cite{Doan_two_time_SA2020} (with some simple modification) also implies that the two-time-scale SA in \eqref{alg:xy} has a convergence rate $\Ocal(1/k)$. While this assumption satisfies in some settings, e.g., see examples presented in \cite{han2024finite}, it may not hold in general (e.g., $H(y) = |y|$ where $y\in\Rset$). On the other hand, in this paper we propose a new approach to achieve an $\Ocal(1/k)$ convergence under a more general condition on $H$, i.e., $H$ is Lipschitz continuous. Finally, we note that when $H$ is linear as presented in Remark \ref{remark:assumption}, the high-order smooth condition above is satisfied. Thus, one can utilize the techniques in \cite{han2024finite, Doan_two_time_SA2020} to show an $\Ocal(1/k)$ convergence rate of \eqref{alg:xy}. This can provide an alternative approach to achieve the same convergence rate as studied in \cite{Kaledin_two_time_SA2020,haque2023tight} for the linear setting. 
\end{remark}

\end{proof}

\section{Simulations}\label{sec:simulations}
In this section, we use the two-time-scale SA to study reinforcement learning algorithms. In particular, we apply the development of fast two-time-scale SA in this paper to design new variants of gradient TD learning in \cite{Sutton2009a} for solving policy evaluation problems in Markov decision processes (MDPs) and online actor-critic methods for solving linear quadratic control (LQR) problems. In our simulations, we will show that the proposed methods converge faster than the existing methods, which agree with our theoretical results.

\subsection{Gradient TD under Linear Function Approximation}
In this section, we will apply different variants of TD learning to solve the policy evaluation problems under linear function approximation in an off-policy setting. In particular, we will evaluate the performance of TD learning (originally proposed in \cite{Sutton1988_TD}, its gradient counterpart (gradient temporal difference with correction or TDC \cite{Sutton2009a}), and the fast TDC based on the proposed approach of \eqref{alg:x-accelerated} and \eqref{alg:y-accelerated} in this paper. For gradient TD learning, we will follow the setting presented in Section \ref{sec:motivating_applications} where $V_{y} = \Phi y$, a linear function of the feature $\Phi$ as studied in \cite{Sutton2009a}. In this case, the gradient TD learning method is a variant of linear two-time-scale SA.  

For our simulation, we generate completely random discounted MDPs with transition probability matrix $\Pcal\in\mathbb{R}^{|\Scal|\times|\Acal|\times |\Scal|}$, $|\Scal|=|\Acal|=50$, drawn i.i.d. from $\text{Unif}(0,1)$ and then normalized such that
\begin{align*}
    \sum_{s'\in\Scal}\Pcal(s'\mid s,a)=1,\quad\forall s,a.
\end{align*}
The behavior policy $\pi_b$ to generate data for our learning is chosen to be uniform for all states. Using the generated data by $\pi_{b}$ we aim to evaluate a softmax policy $\pi$ defined as 
\begin{align*}
    \pi(a\mid s)=\frac{\exp(\psi_{s,a})}{\sum_{a'}\exp(\psi_{s,a'})},\quad\forall s,a,
\end{align*}
where $\psi\in\mathbb{R}^{|\Scal|\times|\Acal|}$ is drawn randomly such that $\psi_{s,a}\sim N(0,1)$. 
We draw the feature matrix $\Phi\in\mathbb{R}^{|\Scal|\times d}$,  $d=10$.
, entry-wise i.i.d. from $N(0,1)$.  The value function $V\in\mathbb{R}^{|\Scal|}$ is the solution to the Bellman equation 
\begin{align*}
    V^{\pi} = R + \gamma P^{\pi} V,
\end{align*}
where the discount factor $\gamma$ is 0.5. We design the experiments such that the value function $V\in\mathbb{R}^{|\Scal|}$ lies in the span of the feature matrix, i.e., there exists a point $\theta^{\star}$ s.t. $V^{\pi} = \Phi\theta^{\star}$. The goal of TD learning is to find $\theta^{\star}$ based on the data $\{s_{k},a_{k},r_{k}\}$ generated by the behavior policy $\pi_{b}$.

For our implementation, the step sizes $\alpha_k$ and $\beta_k$ are set to constant values $5e-4$ and $2e-3$ for simplicity, while the step size $\lambda_k$ is selected as
\[\lambda_k=\frac{4}{5(k+10)}\cdot\]
Figure \ref{fig:random_policy_eval} shows the performance of the three variants of TD learning in our simulation. We observe that although the three variants achieve the same theoretical convergence rate $\Ocal(1/k)$ under linear function approximation, the proposed fast TDC performs slightly better than the other two in this simple simulation. It suggests that estimating the noisy operators before updating the main iterates does accelerate the convergence of the algorithms.

\begin{figure}
\centering
\begin{minipage}{\columnwidth}
  \centering
  \includegraphics[width=\linewidth]{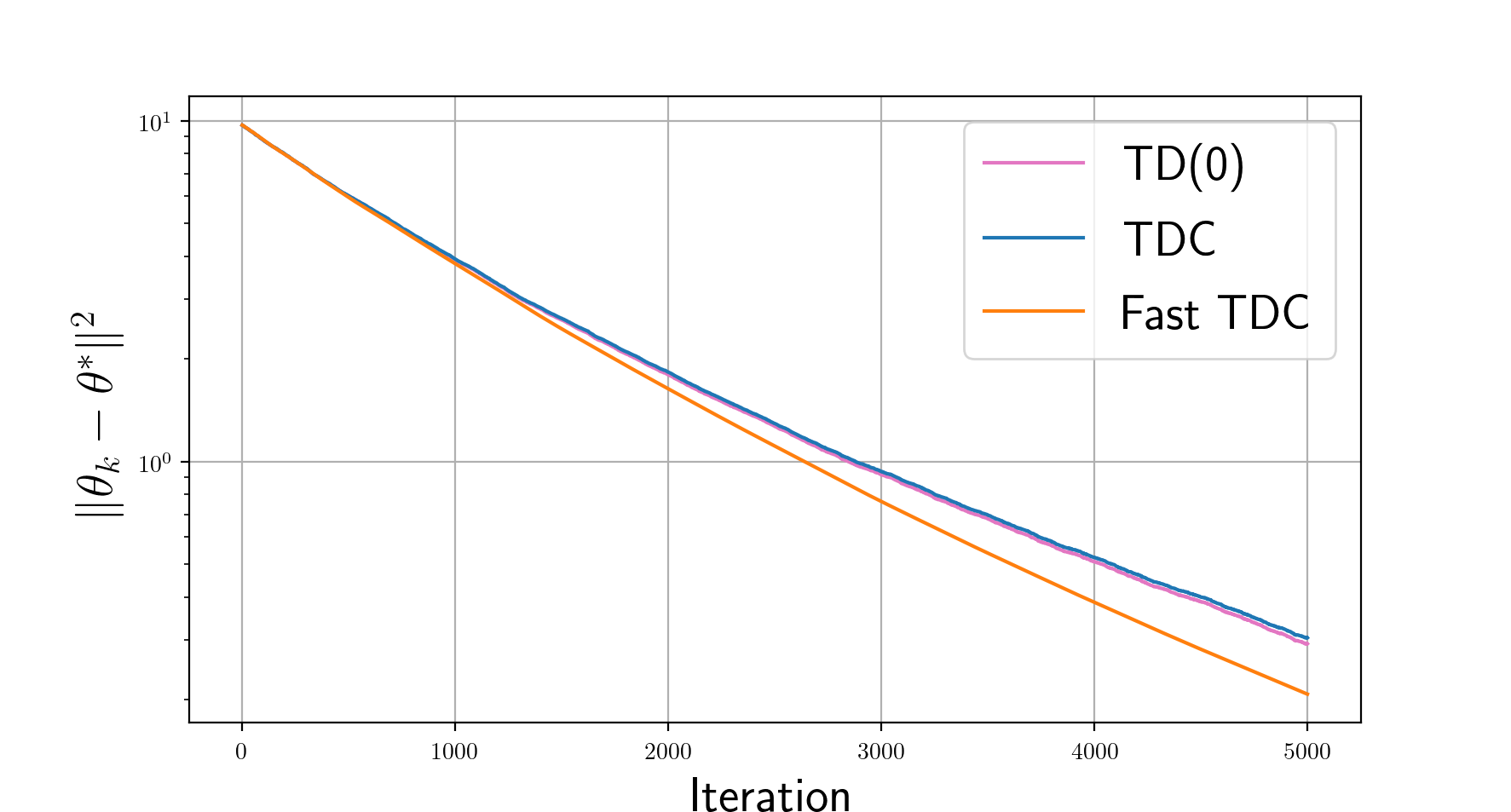}
  \caption{Fast TDC Algorithm for Random Policy Evaluation}
  \label{fig:random_policy_eval}
\end{minipage}%
\begin{minipage}{.08\textwidth}
\hspace{0pt}
\end{minipage}
\begin{minipage}{\columnwidth}
  \centering
  \includegraphics[width=\columnwidth]{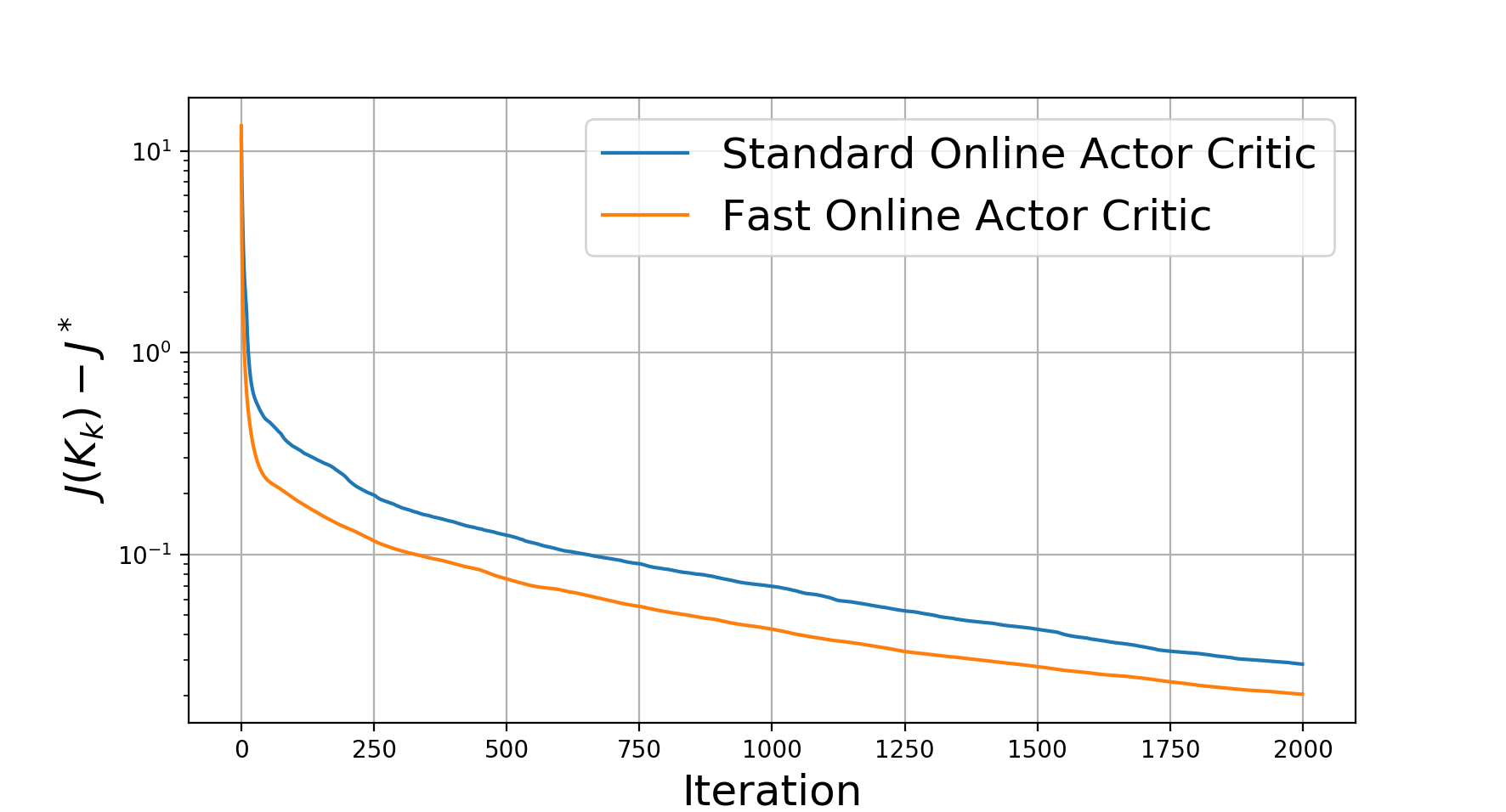}
  \caption{Performance of Fast Actor-Critic Algorithm for LQR}
  \label{fig:policy_optimization_LQR}
\end{minipage}
\vspace{-10pt}
\end{figure}

\subsection{Online Actor-Critic Methods for LQR.} 
We next apply the so-called online actor-critic method, an online variant of policy gradient methods to solve the following infinite-horizon average-cost LQR problem
\begin{align*}
&\underset{\{u_{k}\}}{\text{minimize}}\quad \lim_{T\rightarrow\infty}\frac{1}{T}\mathbb{E}\Big[\sum_{k=0}^{T}\left(x_{k}^{\top} Q x_{k}+u_{k}^{\top} R u_{k}\right)\mid x_{0}\Big] \\
&\text{subject to}\quad  x_{k+1}=A x_{k}+B u_{k}+ w_k,
\end{align*}
where $x_k\in\mathbb{R}^{d_1}$, $u_k\in\mathbb{R}^{d_2}$ are the state and the control variables, $w_k\sim N(0,\Psi)\in\mathbb{R}^{d_1}$ is time-invariant system noise, $A\in\mathbb{R}^{d_1\times d_1}$ and  $B\in\mathbb{R}^{d_1\times d_2}$ are the system transition matrices, and $Q\in\mathbb{R}^{d_1\times d_1}, R\in\mathbb{R}^{d_2\times d_2}$ are positive-definite cost matrices. 
It is well-known (see, for example, \cite[Chap.\ 3.1]{bertsekas2012dynamic}) that the optimal control sequence $\{u_k\}$ that solves the LQR problem is a time-invariant linear function of the state 
\begin{align*}
    u_k^\star = -K^{\star} x_k,\label{eq:lqr_u*}
\end{align*}
where $K^{\star}\in\mathbb{R}^{d_2 \times d_1}$ is a matrix that depends on the problem parameters $A,B,Q,R$. This fact will allow us to reformulate the LQR as an optimization program over the feedback gain matrix $K$, i.e., we aim to search over a set of stochastic linear state feedback controller
\begin{align*}
u_{k} = -K x_{k} + \sigma \epsilon_{k}, \quad \epsilon_{k} \sim \text{N}(0,\sigma^2\mathbf{I}),
\end{align*}
with $\sigma \geq 0$. In particular,  by defining $\Psi_{\sigma}=\Psi+\sigma^2 BB^{\top}$, we can re-express the LQR problem as
\begin{align*}
\underset{K}{\text{minimize }} & J(K)\triangleq \operatorname{trace}(P_{K}\Psi_{\sigma})+\sigma^2\operatorname{trace}(R)\\
\text{s.t. }&P_{K}=Q+K^{\top} R K+(A-B K)^{\top} P_{K}(A-B K).
\end{align*}
We will apply the online actor-critic method studied in \cite{zeng2021two} to solve the optimization problem above, a variant of \eqref{alg:xy}. Moreover, in the context of LQR, the underlying operators of actor-critic methods satisfy the assumptions in \ref{sec:preliminaries}. In \cite{zeng2021two}, the authors show that the online actor-critic method converges at a rate $\Ocal(1/k^{2/3})$. On the other hand, by using the proposed development in this paper, online actor-critic method converges at a better rate $\Ocal(1/k)$.   

For our simulation, we will use the same setting as studied in \cite{zeng2021two}, where the transition and costs matrices are 
\begin{align}
&A=\left[\begin{array}{ccc}
0.5 & 0.01 & 0 \\
0.01 & 0.5 & 0.01 \\
0 & 0.01 & 0.5
\end{array}\right], \quad B=\left[\begin{array}{ll}
1 & 0.1 \\
0 & 0.1 \\
0 & 0.1
\end{array}\right],\notag\\ 
&Q=I_{3},\quad R=I_{2}.
\end{align}
The result of this simulation is shown in Figure \ref{fig:policy_optimization_LQR}, where the proposed fast online actor-critic method outperforms the one studied in \cite{zeng2021two}.

\section{Concluding Remarks}
In this paper, by leveraging the Ruppert-Polyak averaging techniques we propose to develop a new variant of the classic two-time-scale SA that achieves a finite-time convergence rate $\Ocal(1/k)$ in mean-squared regime. Our result significantly improves the existing convergence rate of two-time-scale SA, which is $\Ocal(1/k^{2/3})$. The key idea of our approach is to apply the averaging technique to estimate the operators from their samples before updating the main iterates. This can help to reduce the impact of the sampling noise on the updates of the iterates, therefore, improving its convergence towards the desired solutions. An interesting future direction of this paper is to apply the proposed method to improve the performance of existing control, optimization, and learning algorithms that can be represented as variants of two-time-scale SA.

\section*{References}
\bibliographystyle{IEEEtran}
\bibliography{refs}

\section{Appendix}



\subsection{Proof of Lemma \ref{lem:inequalities}}
Using the Lipschitz continuity of $G,H$, \eqref{alg:xyhat}, and $G(H(y^{\star}),y^{\star}) = 0$ we first have 
\begin{align*}
\|g_{k}\| &\leq \|\Delta g_{k} + G(x_{k},y_{k}) - G(H(y_{k}),y_{k})\|  + \|G(H(y_{k}),y_{k}) - G(H(y^{\star}),y^{\star})\|\leq \|\Delta g_{k}\| + L\|\xhat_{k}\| + L(1+L)\|\yhat_{k}\|.
\end{align*}
Second, since $F(H(y_{k}),y_k) = 0$ we have
\begin{align*}
  \|f_{k}\| &= \|\Delta f_{k} + F(x_{k},y_{k}) - F(H(y_{k}),y_k)\|\notag\\ 
  &\leq \|\Delta f_{k}\| + L\|\xhat_{k}\|. 
\end{align*}
Finally, by using Eqs.\ \eqref{alg:x-accelerated}, \eqref{alg:y-accelerated}, and the relations above we immediately obtain the last two inequalities in \eqref{lem:inequalities:Ineq}.  

\subsection{Proof of Lemma \ref{lem:Delta-f}}
Recall that $\Delta f_{k} = f_{k} - F(x_{k},y_{k})$. Using \eqref{alg:x-accelerated} we consider
\begin{align*}
\Delta f_{k+1}&= (1-\lambda_{k})f_{k} + \lambda_{k} F(x_{k},y_{k}) - F(x_{k+1},y_{k+1}) + \lambda_{k}\epsilon_{k}\\
&= (1-\lambda_{k})\Delta f_{k} + F(x_{k},y_{k}) - F(x_{k+1},y_{k+1}) + \lambda_{k}\epsilon_{k},
\end{align*}
which by taking the quadratic norm on both sides gives
\begin{align}
\|\Delta f_{k+1}\|^2&= (1-\lambda_{k})^2\|\Delta f_{k}\|^2 + 2(1-\lambda_{k})\lambda_{k}\epsilon_{k}^T\Delta f_{k} \notag\\ 
&\quad  + 2(1-\lambda_{k})\Delta f_{k}^T(F(x_{k},y_{k}) - F(x_{k+1},y_{k+1}))\notag\\
&\quad + \| F(x_{k},y_{k}) - F(x_{k+1},y_{k+1}) + \lambda_{k}\epsilon_{k}\|^2\notag\\   
&\leq (1-\lambda_{k})^2\|\Delta f_{k}\|^2 + 2(1-\lambda_{k})\lambda_{k}\epsilon_{k}^T\Delta f_{k}\notag\\  
&\quad  + \frac{\lambda_{k}}{2}\|\Delta f_{k}\|^2 + \frac{2}{\lambda_{k}}\|F(x_{k},y_{k}) - F(x_{k+1},y_{k+1})\|^2\notag\\
&\quad + 2\| F(x_{k},y_{k}) - F(x_{k+1},y_{k+1})\|^2 + 2\lambda_{k}^2\|\epsilon_{k}\|^2\notag\\   
&\leq (1-\lambda_{k})\|\Delta f_{k}\|^2 - \Big(\frac{1}{2} - \lambda_{k}\Big)\lambda_{k} \|\Delta f_{k}\|^2  + 2\lambda_{k}^2\|\epsilon_{k}\|^2 + 2(1-\lambda_{k})\lambda_{k}\epsilon_{k}^T\Delta f_{k}\notag\\   
&\quad  + \frac{4}{\lambda_{k}}\|F(x_{k},y_{k}) - F(x_{k+1},y_{k+1})\|^2,\label{lem:Delta-f:Eq1}
\end{align}
where the second inequality is due the Cauchy-Schwarz inequality and $\lambda_{k} < 1$. Since $F$ satisfies Eq.\ \eqref{assump:smooth:ineq}, the last term on the right-hand side of \eqref{lem:Delta-f:Eq1} is bounded as
\begin{align*}
&\|F(x_{k},y_{k}) - F(x_{k+1},y_{k+1})\|^2 \leq L^2\big(\|x_{k} - x_{k+1}\|^2 + \|x_{k+1}- y_{k+1}\|^2\big)  \notag\\
&\leq L^2\alpha_{k}^2\big(\|\Delta f_{k}\| + L\|\xhat_{k}\| \big)^2  + L^2\beta_{k}^2\big(\|\Delta g_{k}\| + L\|\xhat_{k}\| + L(1+L)\|\yhat_{k}\| \big)^2 \notag\\
&\leq 2L^2\alpha_{k}^2\|\Delta f_{k}\|^2 + 3L^2\beta_{k}^2\|\Delta g_{k}\|^2 + 5L^4\alpha_{k}^2\|\xhat_{k}\|^2 + 3L^4(1+L)^2\beta_{k}^2\|\yhat_{k}\|^2,
\end{align*}
where the second inequality is due to \eqref{lem:inequalities:Ineq}. Substituting the relation above into \eqref{lem:Delta-f:Eq1} and using $\lambda_{k} \leq 1/4$ we obtain \eqref{lem:Delta-f:ineq}
\begin{align*}
\|\Delta f_{k+1}\|^2&\leq (1-\lambda_{k})\|\Delta f_{k}\|^2 - \Big(\frac{1}{2} - \lambda_{k}\Big)\lambda_{k} \|\Delta f_{k}\|^2  + 2\lambda_{k}^2\|\epsilon_{k}\|^2\notag\\
&\quad + 2(1-\lambda_{k})\lambda_{k}\epsilon_{k}^T\Delta f_{k} + \frac{8L^2\alpha_{k}^2}{\lambda_{k}}\|\Delta f_{k}\|^2 \notag\\ 
&\quad  + \frac{12L^2\beta_{k}^2}{\lambda_{k}}\|\Delta g_{k}\|^2 + \frac{15L^4\alpha_{k}^2}{\lambda_{k}}\|\xhat_{k}\|^2  + \frac{12L^4(1+L)^2\beta_{k}^2}{\lambda_{k}}\|\yhat_{k}\|^2\notag\\
&\leq (1-\lambda_{k})\|\Delta f_{k}\|^2 - \Big(\frac{1}{4} - \frac{8L^2\alpha_{k}^2}{\lambda_{k}^2} \Big)\lambda_{k} \|\Delta f_{k}\|^2  + 2\lambda_{k}^2\|\epsilon_{k}\|^2\notag\\
&\quad + 2(1-\lambda_{k})\lambda_{k}\epsilon_{k}^T\Delta f_{k} + \frac{12L^2\beta_{k}^2}{\lambda_{k}}\|\Delta g_{k}\|^2 \notag\\ 
&\quad   + \frac{20L^4\alpha_{k}^2}{\lambda_{k}}\|\xhat_{k}\|^2 + \frac{12L^4(1+L)^2\beta_{k}^2}{\lambda_{k}}\|\yhat_{k}\|^2.
\end{align*}

\subsection{Proof of Lemma \ref{lem:Delta-g}}
The proof of this lemma is similar to the one of Lemma \ref{lem:Delta-f}. For completeness, we present it here. Recall that $\Delta g_{k} = g_{k} - G(x_{k,y_{k}})$. By \eqref{alg:y-accelerated} we have
\begin{align*}
&\Delta g_{k+1}= (1-\gamma_{k})\Delta g_{k} + G(x_{k},y_{k})-G(x_{k+1},y_{k+1}) + \gamma_{k}\psi_{k},   
\end{align*}
which by taking the quadratic norm on both sides gives
\begin{align}
\|\Delta g_{k+1}\|^2&= (1-\gamma_{k})^2\|\Delta g_{k}\|^2  + 2(1 - \gamma_{k})\gamma_{k}\psi_{k}^T\Delta g_{k} \notag\\   
&\quad + 2(1-\gamma_{k})\Delta g_{k}^T(G(x_{k},y_{k}) - G(x_{k+1},y_{k+1}))\notag\\
&\quad + \| G(x_{k},y_{k}) - G(x_{k+1},y_{k+1}) + \gamma_{k}\psi_{k}\|^2\notag\\
&\leq (1-\gamma_{k})^2\|\Delta g_{k}\|^2  + 2(1 - \gamma_{k})\gamma_{k}\psi_{k}^T\Delta g_{k} \notag\\   
&\quad + \frac{\gamma_{k}}{2}\|\Delta g_{k}\|^2 + \frac{2}{\gamma_{k}}\|G(x_{k},y_{k}) - G(x_{k+1},y_{k+1})\|^2\notag\\
&\quad  + 2\| G(x_{k},y_{k}) - G(x_{k+1},y_{k+1})\|^2 + 2\gamma_{k}^2\|\psi_{k}\|^2\notag\\
&\leq (1-\gamma_{k})\|\Delta g_{k}\|^2 - \Big(\frac{1}{2} - \gamma_{k}\Big)\gamma_{k} \|\Delta g_{k}\|^2  + 2\gamma_{k}^2\|\psi_{k}\|^2 + 2(1-\gamma_{k})\gamma_{k}\psi_{k}^T\Delta g_{k}\notag\\   
&\quad  + \frac{4}{\gamma_{k}}\|G(x_{k},y_{k}) - G(x_{k+1},y_{k+1})\|^2,\label{lem:Delta-g:Eq1}
\end{align}
where the last inequality we use $\gamma_{k} < 1$. Since $F$ satisfies Eq.\ \eqref{assump:smooth:ineq}, the last term on the right-hand side of \eqref{lem:Delta-g:Eq1} is bounded as
\begin{align*}
\|G(x_{k},y_{k}) - G(x_{k+1},y_{k+1})\|^2&\leq L^2\big(\|x_{k} - x_{k+1}\|^2 + \|x_{k+1}- y_{k+1}\|^2\big)  \notag\\
&\leq L^2\alpha_{k}^2\big(\|\Delta f_{k}\| + L\|\xhat_{k}\| \big)^2  + L^2\beta_{k}^2\big(\|\Delta g_{k}\| + L\|\xhat_{k}\| + (1+L)^2\|\yhat_{k}\| \big)^2 \notag\\
&\leq 2L^2\alpha_{k}^2\|\Delta f_{k}\|^2 + 3L^2\beta_{k}^2\|\Delta g_{k}\|^2 + 5L^4\alpha_{k}^2\|\xhat_{k}\|^2  + 3L^4(1+L)^2\beta_{k}^2\|\yhat_{k}\|^2,
\end{align*}
which when substituting  into \eqref{lem:Delta-g:Eq1} and using $\beta_{k} \leq 1/4$ we obtain \eqref{lem:Delta-g:ineq}
\begin{align*}
\|\Delta g_{k+1}\|^2&\leq (1-\gamma_{k})\|\Delta g_{k}\|^2 - \Big(\frac{1}{2} - \gamma_{k}\Big)\gamma_{k} \|\Delta g_{k}\|^2  + 2\gamma_{k}^2\|\psi_{k}\|^2 + 2(1-\gamma_{k})\gamma_{k}\psi_{k}^T\Delta g_{k}  + \frac{12L^2\beta_{k}^2}{\gamma_{k}}\|\Delta g_{k}\|^2\notag\\   
&\quad  + \frac{8L^2\alpha_{k}^2}{\gamma_{k}}\|\Delta f_{k}\|^2  + \frac{20L^4\alpha_{k}^2}{\gamma_{k}}\|\xhat_{k}\|^2 + \frac{12L^4(1+L)^2\beta_{k}^2}{\gamma_{k}}\|\yhat_{k}\|^2\notag\\
&\leq (1-\gamma_{k})\|\Delta g_{k}\|^2 - \Big(\frac{1}{4} - \frac{12L^2\beta_{k}^2}{\gamma_{k}^2}\Big)\gamma_{k} \|\Delta g_{k}\|^2  + 2\gamma_{k}^2\|\psi_{k}\|^2+ 2(1-\gamma_{k})\gamma_{k}\psi_{k}^T\Delta g_{k}\notag\\
&\quad  + \frac{8L^2\alpha_{k}^2}{\gamma_{k}}\|\Delta f_{k}\|^2     + \frac{20L^4\alpha_{k}^2}{\gamma_{k}}\|\xhat_{k}\|^2  + \frac{12L^4(1+L)^2\beta_{k}^2}{\gamma_{k}}\|\yhat_{k}\|^2.
\end{align*}

\subsection{Proof of Lemma \ref{lem:xhat}}
Using \eqref{alg:x-accelerated} we consider 
\begin{align*}
&\xhat_{k+1} = x_{k+1} - H(y_{k+1}) = x_{k} - \alpha_{k}f_{k} - H(y_{k+1})\notag\\
&= \xhat_{k} - \alpha_{k}F(x_{k},y_{k}) - \alpha_{k}\Delta f_{k} +  H(y_{k}) - H\left(y_{k}-\beta_{k}g_{k}\right),
\end{align*}
which gives
\begin{align}
\|\xhat_{k+1}\|^2 &= \left\|\xhat_{k} - \alpha_{k}F(x_{k},y_{k}) - \alpha_{k}\Delta f_{k} +  H(y_{k}) - H\left(y_{k}-\beta_{k}g_{k}\right)\right\|^2\notag\\
&= \|\xhat_{k} - \alpha_{k}F(x_{k},y_{k})\|^2 + \left\|H(y_{k}) - H(y_{k}-\beta_{k}g_{k})\right\|^2\notag\\ 
&\quad + \alpha_{k}^2\|\Delta f_{k}\|^2 - 2\alpha_{k}\left(\xhat_{k} - \alpha_{k}F(x_{k},y_{k})\right)^T\Delta f_{k}\notag\\
&\quad + 2\left(\xhat_{k} - \alpha_{k}F(x_{k},y_{k})\right)^T\left(H(y_{k}) - H(y_{k}-\beta_{k}g_{k})\right)\notag\\
&\quad - 2\alpha_{k}\left(H(y_{k}) - H(y_{k}-\beta_{k}g_{k})\right)^T\Delta f_{k}\notag\\
&= A_{1} + A_{2} + \alpha_{k}^2\|\Delta f_{k}\|^2 + A_{3} + A_{4} + A_{5}, \label{lem:xhat:Eq1} 
\end{align}
where $A_{1},\ldots,A_{5}$ are the terms on the right-hand side of the second equation, in that order. We next analyze these terms. First, using $F(H(y_{k}),y_{k}) = 0$ and $\xhat_{k} = x_{k} - H(y_{k})$ we have 
\begin{align}
A_{1} &= \|\xhat_{k}  - \alpha_{k}F(x_{k},y_{k})\|^2\notag\\ 
&= \|\xhat_{k}\|^2 - 2\alpha_{k}\xhat_{k}^TF(x_{k},y_{k}) + \| \alpha_{k}F(x_{k},y_{k})\|^2\notag\\
&= \|\xhat_{k}\|^2 - 2\alpha_{k}(x_{k} - H(y_{k}) )^T\left(F(x_{k},y_{k}) - F(H(y_{k}),y_{k})\right) + \alpha_{k}^2\| F(x_{k},y_{k}) - F(H(y_{k}),y_{k})\|^2\notag\\
&\leq \|\xhat_{k}\|^2 - 2\mu_{F}\alpha_{k}\|x_{k}-H(y_{k})\|^2 + L^2\alpha_{k}^2\|x_{k} - H(y_{k})\|^2\notag\\ 
&= \left(1-2\mu_{F}\alpha_{k} + L^2\alpha_{k}^2\right) \|\xhat_{k}\|^2,\label{lem:xhat:Eq1a}
\end{align}
where the first inequality is due to the strong monotonicity and Lispchitz continuity of $F$. Second, using the Lipschitz continuity of $H$ and \eqref{lem:inequalities:Ineq} we obtain
\begin{align}
A_{2} &=  \left\|H(y_{k}) - H(y_{k}-\beta_{k}g_{k})\right\|^2 \leq L^2\beta_{k}^2\|g_{k}\|^2\notag\\ 
&\leq  L^2\beta_{k}^2\big(\|\Delta g_{k}\| + L\|\xhat_{k}\| + L(1+L)\|\yhat_{k}\|\big)^2\notag\\
&\leq 3L^2\beta_{k}^2\big(\|\Delta g_{k}\|^2 + L^2\|\xhat_{k}\|^2 + L^2(1+L)^2\|\yhat_{k}\|^2\big). \label{lem:xhat:Eq1b}
\end{align}
Third, using \eqref{lem:xhat:stepsizes} into \eqref{lem:xhat:Eq1a} we obtain
\begin{align*}
\|\xhat_{k}  - \alpha_{k}F(x_{k},y_{k})\| &\leq \sqrt{1-2\mu_{F}\alpha_{k} + L^2\alpha_{k}^2}\|\xhat_{k}\|\leq \|\xhat_{k}\|,
\end{align*}
which gives 
\begin{align}
A_{3} &=  -2\alpha_{k}\left(\xhat_{k} - \alpha_{k}F(x_{k},y_{k})\right)^T\Delta f_{k}\leq 2\alpha_{k}\|\xhat_{k} - \alpha_{k}F(x_{k},y_{k})\|\|\Delta f_{k}\|\notag\\
&\leq 2\alpha_{k}\|\xhat_{k}\|\|\Delta f_{k}\|.\label{lem:xhat:Eq1d}
\end{align}
Similarly, using $F(H(y_k),y_k) = 0$ we consider $A_{4}$
\begin{align}
A_{4} &= 2\left(\xhat_{k} - \alpha_{k}F(x_{k},y_{k})\right)^T\left(H(y_{k}) - H(y_{k}-\beta_{k}g_{k})\right)\leq 2L\beta_{k}\|\xhat_{k} - \alpha_{k}F(x_{k},y_{k})\|\|g_{k}\|\notag\\
&\leq 2L\beta_{k}\|\xhat_{k}\|\|g_{k}\|\leq 2L\beta_{k}\|\xhat_{k}\|\big(\|\Delta g_{k}\| + L\|\xhat_{k}\| + L(1+L)\|\yhat_{k}\|\big)\notag\\
&= 2L^2\beta_{k}\|\xhat_{k}\|^2 + 2L\beta_{k}\|\xhat_{k}\|\|\Delta g_{k}\| + 2L^2(1+L)\beta_{k}\|\xhat_{k}\|\|\yhat_{k}\|.\label{lem:xhat:Eq1c}
\end{align}
Finally, using \eqref{lem:inequalities:Ineq} one more time  we have
\begin{align}
A_{5} &= - 2\alpha_{k}\left(H(y_{k}) - H(y_{k}-\beta_{k}g_{k})\right)^T\Delta f_{k}\leq 2L\alpha_{k}\beta_{k}\|g_{k}\|\|\Delta f_{k}\|\notag\\ 
&\hspace{-0.2cm}\leq 2L\alpha_{k}\beta_{k}\big(\|\Delta g_{k}\| + L\|\xhat_{k}\| + L(1+L)\|\yhat_{k}\|\big)\|\Delta f_{k}\|.\label{lem:xhat:Eq1e}
\end{align}
Substituting Eqs.\ \eqref{lem:xhat:Eq1a}--\eqref{lem:xhat:Eq1e} into \eqref{lem:xhat:Eq1} we obtain \eqref{lem:xhat:Ineq}, i.e.,
\begin{align*}
\|\xhat_{k+1}\|^2&\leq  \left(1-2\mu_{F}\alpha_{k} + L^2\alpha_{k}^2\right) \|\xhat_{k}\|^2 + \alpha_{k}^2\|\nabla f_{k}\|^2\notag\\
&\quad + 3L^2\beta_{k}^2\big(\|\Delta g_{k}\|^2 + L^2\|\xhat_{k}\|^2 + L^2(1+L)^2\|\yhat_{k}\|^2\big)\notag\\
&\quad + 2\alpha_{k}\|\xhat_{k}\|\|\Delta f_{k}\|+ 2L^2\beta_{k}\|\xhat_{k}\|^2\notag\\ 
&\quad + 2L\beta_{k}\|\xhat_{k}\|\|\Delta g_{k}\| + 2L^2(1+L)\beta_{k}\|\xhat_{k}\|\|\yhat_{k}\|\notag\\
&\quad +2L\alpha_{k}\beta_{k}\big(\|\Delta g_{k}\| + L\|\xhat_{k}\| + L(1+L)\|\yhat_{k}\|\big)\|\Delta f_{k}\|\notag\\
&\leq  \left(1-2\mu_{F}\alpha_{k} + L^2\alpha_{k}^2\right)  + \alpha_{k}^2\|\Delta f_{k}\|^2\notag\\ 
&\quad +  3L^2\beta_{k}^2\big(\|\Delta g_{k}\|^2 + L^2\|\xhat_{k}\|^2 + L^2(1+L)^2\|\yhat_{k}\|^2\big)\notag\\
&\quad + \frac{\mu_{F}\alpha_{k}}{2}\|\xhat_{k}\|^2 + \frac{2}{\mu_{F}}\alpha_{k}\|\nabla f_{k}\|^2 + 2L^2\beta_{k}\|\xhat_{k}\|^2 \notag\\ 
&\quad + L^2\beta_{k}\|\xhat_{k}\|^2 + \beta_{k}\|\Delta g_{k}\|^2\notag\\
&\quad + \frac{4L^4(1+L)^2}{\mu_{G}}\beta_{k}\|\xhat_{k}\|^2 + \frac{\mu_{G}\beta_{k}}{4}\|\yhat_{k}\|^2 \notag\\
&\quad  + L^4\alpha_{k}\beta_{k}\|\xhat_{k}\|^2 + 3\alpha_{k}\beta_{k}\|\Delta f_{k}\|^2 \notag\\ 
&\quad + L^2\alpha_{k}\beta_{k}\big(\|\Delta g_{k}\|^2 + L^2(1+L)^2\|\|\yhat_{k}\|^2\big)\notag\\
&\leq \left(1 -\mu_{F}\alpha_{k}\right)\|\xhat_{k}\|^2 + L^2\big(4L^2 + 1\big)\alpha_{k}^2\|\xhat_{k}\|^2\notag\\ 
&\quad - \left[\frac{\mu_{F}\alpha_{k}}{2} - \frac{4L^3\big(1+L(1+L)^2\big)\beta_{k}}{\mu_{G}}\right]\|\xhat_{k}\|^2\notag\\
&\quad + \frac{\mu_{G}\beta_{k}}{4}\|\yhat_{k}\|^2 + 4L^4(1+L)^{2} \alpha_{k}\beta_{k}\|\yhat_{k}\|^2\notag\\
&\quad + \left[\frac{2\alpha_{k}}{\mu_{F}} + 4\alpha_{k}^2\right]\|\Delta f_{k}\|^2 + \left[\beta_{k} + 4L^2\alpha_{k}\beta_{k} \right]\|\Delta g_{k}\|^2,
\end{align*}
where the last inequality we use $\beta_{k}\leq \alpha_{k}$ and $\mu_{G}\leq L$.

\subsection{Proof of Lemma \ref{lem:yhat}}
Using \eqref{alg:y-accelerated} we consider 
\begin{align*}
\yhat_{k+1} &= y_{k+1} - y^{\star} = \yhat_{k} - \beta_{k}g_{k}\notag\\ 
&= \yhat_{k} - \beta_{k}G(H(y_{k}),y_{k}) +\beta_{k}(G(H(y_{k}),y_{k}) - G(x_{k},y_{k})) - \beta_{k}\Delta g_{k},
\end{align*}    
which gives
\begin{align}
\|\yhat_{k+1}\|^2&= \|\yhat_{k} - \beta_{k}G(H(y_{k}),y_{k})\|^2 + \|\beta_{k}(G(H(y_{k}),y_{k}) - G(x_{k},y_{k})) - \beta_{k}\Delta g_{k}\|^2\notag\\
&\quad + 2\beta_{k}(\yhat_{k} - \beta_{k}G(H(y_{k}),y_{k})^T(G(H(y_{k}),y_{k}) - G(x_{k},y_{k}))\notag\\
&\quad - 2\beta_{k}(\yhat_{k} - \beta_{k}G(H(y_{k}),y_{k})^T\Delta g_{k}\notag\\
&\leq  \|\yhat_{k} - \beta_{k}G(H(y_{k}),y_{k})\|^2 + 2\beta_{k}^2\|\Delta g_{k}\|^2 + 2\beta_{k}^2\|G(H(y_{k}),y_{k}) - G(x_{k},y_{k})\|^2\notag\\ 
&\quad + 2\beta_{k}(\yhat_{k} - \beta_{k}G(H(y_{k}),y_{k})^T(G(H(y_{k}),y_{k}) - G(x_{k},y_{k}))\notag\\
&\quad - 2\beta_{k}(\yhat_{k} - \beta_{k}G(H(y_{k}),y_{k})^T\Delta g_{k}.\label{lem:yhat:Eq1}
\end{align}
First, using the strong monotonicity of $G$ and $G(H(y^{\star}),y^{\star}) = 0$ we have
\begin{align}
\|\yhat_{k} - \beta_{k}G(H(y_{k}),y_{k})\|^2 &= \|\yhat_{k}\|^2 + \beta_{k}^2\|G(H(y_{k}),y_{k})-G(H(y^{\star}),y^{\star})\|^2 - 2\beta_{k}(y_{k}-y^{\star})^TG(x_{k},y_{k})\notag\\
&\leq (1 - 2\mu_{G}\beta_{k} + L^2(L+1)^2\beta_{k}^2)\|\yhat_{k}\|^2.\label{lem:yhat:Eq1a}
\end{align}
Next, since  $\beta_{k} \leq \mu_{G}/(4L^2(1+L)^2)$ we have from the preceding relation that $\|\yhat_{k} - \beta_{k}G(H(y_{k}),y_{k})\| \leq \|\yhat_{k}\|$. Thus, using \eqref{lem:yhat:Eq1a} we consider
\begin{align}
&2\beta_{k}(\yhat_{k} - \beta_{k}G(H(y_{k}),y_{k})^T(G(H(y_{k}),y_{k}) - G(x_{k},y_{k}) - \Delta g_{k})\notag\\ 
&\leq 2\beta_{k} \|\yhat_{k} - \beta_{k}G(H(y_{k}),y_{k})\|(L\|\xhat_{k}\| + \Delta g_{k})\notag\\
&\leq 2\beta_{k}\|\yhat_{k}\|\big(L\|\xhat_{k}\| + \|\Delta g_{k}\|\big)\notag\\
&\leq \frac{\mu_{G}\beta_{k}}{4}\|\yhat_{k}\|^2 + \frac{8\beta_{k}}{\mu_{G}}\big(L^2\|\xhat_{k}\|^2 + \|\Delta g_{k}\|^2\big)
\label{lem:yhat:Eq1b}
\end{align}
Substituting \eqref{lem:yhat:Eq1a} and \eqref{lem:yhat:Eq1b} into \eqref{lem:yhat:Eq1} and using \eqref{lem:yhat:stepsizes} we obtain \eqref{lem:yhat:Ineq}, i.e., 
\begin{align*}
\|\yhat_{k+1}\|^2 &\leq  (1 - 2\mu_{G}\beta_{k} + L^2(L+1)^2\beta_{k}^2)\|\yhat_{k}\|^2 + 2\beta_{k}^2\|\Delta g_{k}\|^2\notag\\ 
&\quad + 2L^2\beta_{k}^2\|\xhat_{k}\|^2 + \frac{\mu_{G}\beta_{k}}{4}\|\yhat_{k}\|^2 + \frac{8\beta_{k}}{\mu_{G}}\big(L^2\|\xhat_{k}\|^2 + \|\Delta g_{k}\|^2\big)\\
&\leq (1 - \mu_{G}\beta_{k})\|\yhat_{k}\|^2 - \left(\frac{3\mu_{G}\beta_{k}}{4} - L^2(L+1)^2\beta_{k}^2 \right)\|\yhat_{k}\|^2\notag\\ 
&\quad + 2L^2\Big(\frac{4\beta_{k}}{\mu_{G}} + \beta_{k}^2\Big)\|\xhat_{k}\|^2 + \left(\frac{8\beta_{k}}{\mu_{G}} + 2\beta_{k}^2\right)  \|\Delta g_{k}\|^2\notag\\
&\leq (1 - \mu_{G}\beta_{k})\|\yhat_{k}\|^2 - \frac{\mu_{G}\beta_{k}}{2}\|\yhat_{k}\|^2  + 2L^2\Big(\frac{4\beta_{k}}{\mu_{G}} + \beta_{k}^2\Big)\|\xhat_{k}\|^2 + \left(\frac{8\beta_{k}}{\mu_{G}} + 2\beta_{k}^2\right)  \|\Delta g_{k}\|^2.
\end{align*}

\end{document}